\documentclass[twoside,final,leqno,onefignum,onetabnum]{siamltex1213}

\usepackage[T1]{fontenc}
\usepackage[utf8]{inputenc}
\usepackage[american]{babel}

\usepackage{tikz}
\usepackage{booktabs}
\usepackage{mathtools}
\usepackage[ruled,vlined]{algorithm2e}
\usepackage{subfig}

\usetikzlibrary
  {arrows,calc,through,backgrounds,matrix,positioning,decorations.pathmorphing}

\usepackage{amsmath}
\usepackage{amssymb}
\usepackage{esint} 
\newtheorem{remark}[theorem]{Remark}

\usepackage[square,numbers,sort&compress]{natbib}

\DeclareMathOperator*{\arginf}{arg\,inf}

\title{%
  A duality-based optimization approach for model adaptivity in
  heterogeneous multiscale problems}

\author{%
Matthias Maier%
\thanks{%
  School of Mathematics, University of Minnesota, %
  206 Church Street SE, Minneapolis, MN 55455, USA. \email{msmaier@umn.edu}.}
\and Rolf Rannacher%
\thanks{%
  Institute of Applied Mathematics, Heidelberg University, Im Neuenheimer %
  Feld 205, 69120 Heidelberg, Germany. %
  \email{rannacher@iwr.uni-heidelberg}}
}

\begin{document}

\maketitle
\slugger{mms}{xxxx}{xx}{x}{x--x}

\begin{abstract}
  This paper introduces a novel framework for model adaptivity in the context
  of heterogeneous multiscale problems. The framework is based on the idea
  to interpret model adaptivity as a \emph{minimization problem} of local
  error indicators, that are derived in the general context of the
  \emph{Dual Weighted Residual} (DWR) method. Based on the optimization
  approach a post-processing strategy is formulated that lifts the
  requirement of strict a priori knowledge about applicability and quality
  of effective models. This allows for the systematic, ``goal-oriented''
  tuning of \emph{effective models} with respect to a \emph{quantity of
  interest}.
  The framework is tested numerically on elliptic diffusion problems with
  different types of heterogeneous, random coefficients, as well as an
  advection-diffusion problem with strong microscopic, random advection
  field.
\end{abstract}

\begin{keywords}
  finite element method, mesh adaptation, model optimization, model
  adaptation, goal-oriented adaptivity, DWR method
\end{keywords}

\begin{AMS}
  35J15 65C20 65N12 65N15 65N30 65N50
\end{AMS}

\pagestyle{myheadings}
\thispagestyle{plain}
\markboth%
  {M. Maier \& R. Rannacher}%
  {Duality-based optimization in multiscale problems}


\section{Introduction}
\label{sec:3introduction}

A number of different approaches for modeling multiscale phenomena in the
context of finite-element methods have been introduced over the last years.
They either rely on the existence of a periodic or stochastic substructure
or on the scale-dependent splitting of variational solution- and test
spaces \cite{Hughes:1998, Brezzi:1999, Efendiev:2004, E:2003a}. The use of
multiscale methods comes at a significant price with respect to sources of
error: Among the usual discretization error due to a numerical
approximation of the partial differential equation (PDE), multiscale
methods exhibit an inherent \emph{model error} resulting from a
\emph{modeling assumption} for scale separation. This makes the idea of a
posteriori error estimation, where a quantitative estimate for the
different sources of error is computed by means of a post-processing
approach highly attractive. The a posteriori control of discretization
errors for multiscale methods is well understood \cite{Henning:2011,
Henning:2012, Abdulle:2013, Larson:2004, Larson:2005}.

The novelty of a posteriori error estimation with respect to multiscale
methods lies in the possibility for \emph{model adaptivity}. First results
for \emph{estimating and controlling the model error} in the context of
multiscale schemes were given by Oden et
al.~\cite{Oden:2000a,Oden:2000b,Romkes:2007} and Braack and
Ern~\cite{Braack:2003} in the context of the \emph{Heterogeneous Multiscale
Method} (HMM) \cite{E:2003a}. The key idea is to use the error-identity
stemming from  the \emph{Dual Weighted Residual} (DWR) method introduced by
Becker and Rannacher~\cite{Becker:1996a, Becker:2001} to quantify a local
model error. This information can then be used for different
model-adaptation strategies: A possible approach is to locally switch from
a cheap, coarse model to an expensive, full model within an adaptation
cycle \cite{Braack:2003}. Alternatively, as a pure post-processing
strategy, \emph{region of influence} can be constructed on which a
finescale correction is computed in full
\cite{Oden:2000a,Oden:2000b,Romkes:2007}.

Based on a multiscale framework introduced in a previous publications by the
authors \cite{Maier:2014,Maier:2015}, this paper presents a novel approach
for model adaptivity. Instead of using an a priori choice of increasingly
accurate models to switch between them (depending on the local error
estimate), it uses the error identity obtained by the duality argument 
directly in a
minimization problem. This has the advantage that no a priori knowledge
about effective models and reconstruction principles has to be available.
The optimization problem itself is used to select the optimal model.

The optimization approach requires a certain quality of the approximation
of the dual solution that will be addressed with an efficient local
reconstruction approach. The error identity lifts the question of suitable
approximation in terms of a \emph{quantity of interest} to the question of
suitable approximation properties of the localization technique for the
dual problem. The latter is typically measured in the $L^2$-norm of the
gradient of the error of the dual approximation, for which---depending on
the localization approach---strong approximation properties are available.
Thus, the proposed optimization framework can be interpreted as a
multiscale method in its own right, where a reconstruction process is used
for the dual solution. The \emph{modeling aspect} of the optimization
problem lies in the choice of the quantity of interest and the choice of
local reconstruction of the dual solution.

The outline of the paper is as follows. In Section~\ref{sec:2mhmm} an
abstract multiscale scheme for model adaptation is outlined shortly
\cite{Maier:2014,Maier:2015}. The section also covers the necessary a
posteriori error analysis with the DWR method and an efficient
approximation strategy for the solution of the dual problem involved. With
this prerequisites at hand, a model-optimization framework is introduced in
Section~\ref{sec:3framework}. It is based on a minimization problem
formulated with the help of an error identity from the a posteriori error
analysis. Implementational details are discussed in
Section~\ref{sec:4implementational}. In Section~\ref{sec:5numerical} an
extensive numerical study for an elliptic diffusion problem and an
advection-diffusion problem is shown. A conclusion and outlook is given in
Section~\ref{sec:6conclusion}.


\section{An abstract multiscale scheme for model adaptation}
\label{sec:2mhmm}

The discussion in this section is based on a multiscale scheme for model
adaptation introduced by the authors \cite{Maier:2014,Maier:2015} that
explicitly decouples all discretization and modeling parameters. It is a
reformulation of the classical HMM method by E and
Engquist~\cite{E:2003a} and shares similarities with model adaptation
frameworks introduced by Oden and Vemaganti~\cite{Oden:2000a, Oden:2000b}
and Braack and Ern~\cite{Braack:2003}. We briefly discuss a slightly
simplified variant in this section. For a detailed introduction we refer to
the aforementioned publications.

Let us consider the following multi-scale model problem: Find
$u^\varepsilon\in H^1_0\left(\Omega\right)$ s.\,t.
\begin{gather}
  \label{eq:3modprob}
  \left(A^\varepsilon\nabla u^\varepsilon, \nabla\varphi \right) =
  \left(f,\varphi \right)\quad\forall\,\varphi\in H^1_0\left(\Omega\right),
\end{gather}
on a bounded domain $\Omega\subset\mathbb{R}^d\;(d=2,3)$ where the
generally tensor-valued function $A^\varepsilon\in L^\infty \left(
\Omega\right)^{d\times d}$ is of heterogeneous character and
highly oscillating on a small length scale indicated by a scaling
parameter ${\varepsilon}$.
Here, $H^1_0\left(\Omega\right)$ is the
usual first-order Sobolev Hilbert space with zero Dirichlet data along the
boundary $ {\partial}\Omega$. $(\cdot,\cdot)$ denotes the $L^2$
scalar product on $\Omega$ and $\|\cdot\| = (\cdot,\cdot)^{1/2}$
the corresponding norm. The norms of other function spaces are indicated by
subscripts, e.\,g., $\|\cdot\|_{L^\infty(\Omega)}$ or $\|\cdot\|_K =
\|\cdot\|_{L^2(K)}$ for a subset $K\subset \bar\Omega$. We assume the
coefficient tensor $A^\varepsilon$ to be symmetric and positive
definite (uniformly in $\varepsilon$),
\begin{align}
  \label{eq:2elliptic}
  A^\varepsilon_{ij} =  A^\varepsilon_{ji},\quad
  \alpha|\xi|^2 \leq \sum_{i,j=1}^d A^\varepsilon_{ij}\xi_i\xi_j \leq
  \beta|\xi|^2,\quad \text{a.\,e. on }\Omega, \quad \xi\in\mathbb{R}^d,
\end{align}
with constants $\,\alpha,\beta\in\mathbb{R}_{+}$, so that
(\ref{eq:3modprob}) admits a unique solution.

Due to the finescale character of $A^\varepsilon$, a direct numerical
simulation of (\ref{eq:3modprob}) is computationally very expensive.
We thus introduce an \emph{effective model problem}
\begin{gather}
  \label{eq:3effectiveproblem}
  \left(\bar A^\delta\nabla u^\delta, \nabla\varphi \right) =
  \left(f,\varphi \right)\quad\forall\,\varphi\in H^1_0\left(\Omega\right)
\end{gather}
based on a sampling mesh $\mathbb{T}_{\delta}(\Omega)$ of $\Omega$ together
with an effective tensor
\begin{align}
  \bar A^\delta:\,\mathbb{T}_{\delta}(\Omega)\to\mathbb{R}^{d\times d}
\end{align}
with region-wise constant values; see Figure~\ref{fig:3omega}.

\begin{remark}
  Effective parameters $\big(\mathbb{T}_{\delta}(\Omega),A^\delta\big)$ can
  be obtained by different means, e.\,g., by using cell problems derived
  within a corresponding homogenization theory
  \cite{Allaire:1992,Cioranescu:1999}:
  \begin{gather}
    \label{eq:3mhmmhomoga}
    \bar A^\delta_{ij}(K)\coloneqq\fint_{Y_K^\delta}
    A^\varepsilon(x)\big(\nabla_x\omega_i(x)+\boldsymbol{e}_i\big)\cdot
    \big(\nabla_x\omega_j(x)+\boldsymbol{e}_j\big)\,\mathrm{d} x,
  \end{gather}
  where $\boldsymbol{e}_i$ denotes the $i$-th cartesian unit vector and the
  $\omega_i\in\tilde H^1_{\text{per}}(Y_K^\delta)$ are solutions of a
  so-called cell problem
  \begin{gather}
    \label{eq:3mhmmhomog}
    \int_{Y_K^\delta}A^\varepsilon(x)\big(\nabla_x\omega_i(x)+\boldsymbol{e}_i)
    \cdot\nabla\varphi=0
    \quad\forall\varphi\in\tilde H^1_{\text{per}}(Y_K^\delta),
  \end{gather}
  or by using simple averaging strategies such as the geometric mean
  value \cite{Warren:1961}:
  \begin{gather}
    \label{eq:3mhmmaverage}
    \log \bar A^\delta_{ij}(K) \coloneqq \frac1{|Y_K^\delta|}\int_{Y_K^\delta}
    \big(\log A^\varepsilon_{ij}(y)\big) \,\mathrm{d} y.
  \end{gather}
  Here, $Y_K^\delta$ denotes a rescaled copy of the unit cell $Y=[0,1]^d$
  centered at the midpoint of a given sampling-mesh cell
  $K\in\mathbb{T}_{\delta}(\Omega)$. $\tilde H^1_{\text{per}}(Y_K^\delta)$
  is the subspace of $H^1(Y_K^\delta)$ consisting of $d$-periodic
  functions with zero mean value. $\fint_K\coloneqq1/|K|\int_K$ denotes the
  arithmetic average.
\end{remark}

The purpose of this paper is to discuss a novel approach of determining the
effective values $\bar A^\delta$ (for a given sampling discretization) by
means of an optimization process. It will be based on an error identity
given by the solution of a dual problem. For the sake of simplicity, we
will neglect \emph{finescale discretization} errors that emerge by
numerically approximating (\ref{eq:3mhmmhomoga}), or
(\ref{eq:3mhmmaverage}), and will only introduce a macroscale
discretization.
%
\begin{figure}
  \centering
    \begin{tikzpicture}[scale=1.2]
      \draw[fill=black!40,thick] (2.0,1.0) rectangle (3.0,2.0);
      \draw[step=0.5,gray,thick] (0.25,0.25) grid (3.75,3.75);
      \draw[step=0.25,gray,thick] (0.26,0.26) grid (2.5,2.5);
      \draw[step=0.125,gray,thick] (0.5,1.25) grid (1.5,2.25);
      \node at (-0.45,0.5) {$\mathbb{T}_{H}(\Omega)$};
      \draw[step=1.0, line width=2.0pt] (0.25,0.25) grid (3.75,3.75);
      \node at (-0.45,3.5) {$\mathbb{T}_{\delta}(\Omega)$};
      \node (A) at (2.5,1.5) {};
      \draw[fill=black!40,thick] (6.5,1.0) rectangle (7.5,2.0);
      \draw[thick] (6.5,1.0) rectangle (7.5,2.0);
      \node[] (B) at (7.0,1.5) {};
      \node[] at (7.0,0.5) {$Y_K^\delta$};
      \draw[<-, thick, shorten <=5pt]
       (A) edge[bend left] node[above]
       {$\qquad\bar A^\delta_K\in\mathbb{R}^{d\times d}$}(B);
    \end{tikzpicture}
  \caption{The computational domain $\Omega$ together with the
    \emph{sampling mesh} $\mathbb{T}_\delta(\Omega)$ consisting of sampling
    regions $K\in\mathbb{T}_\delta$. The coarse mesh $\mathbb{T}_H(\Omega)$ used for the
    final finite-element discretization is a refinement of the sampling
    mesh $\mathbb{T}_\delta$.
  }
  \label{fig:3omega}
\end{figure}
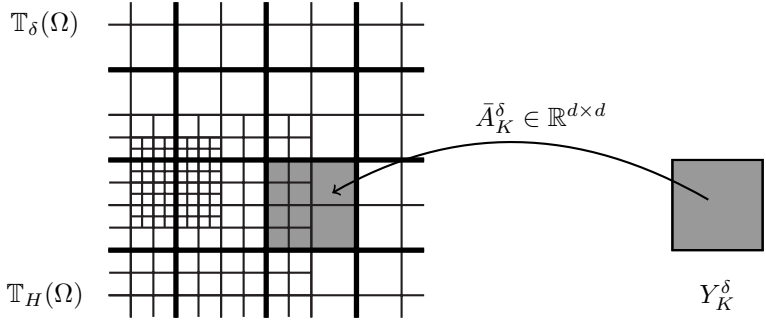

Let $\mathbb{T}_{H}(\Omega)$ be a coarse grid for numerically approximating
the variational equation (\ref{eq:3effectiveproblem}).
%
\begin{definition}[Fully discretized problem]
  \label{defi:3probdisc}
  Let $\mathbb{T}_{H}$ be a mesh covering $\overline \Omega$, and let
  $V_{H}(\Omega)\subset H^1_0(\Omega)$ be a corresponding finite-element ansatz 
  space. The fully discrete problem reads: Find
  $U\in V_{H}(\Omega)$ s.\,t.
  \begin{align}
    \label{eq:3probdisc}
    \left(\bar A^\delta\nabla U, \nabla\varphi^H \right) = \left( f,\varphi^H
    \right)\quad\forall\,\varphi^H\in V_{H}(\Omega).
  \end{align}
\end{definition}
%
\begin{remark}
  The problems (\ref{eq:3modprob}), (\ref{eq:3effectiveproblem}), and
  (\ref{eq:3probdisc}) are well-posed. Further, specific a priori
  assumptions on $A^\varepsilon$ and corresponding concrete choices of
  upscaling (such as (\ref{eq:3mhmmhomoga}), and (\ref{eq:3mhmmaverage}))
  lead to different a priori error estimates. We refer to \cite{Maier:2014,
  Maier:2015}, for a detailed discussion.
\end{remark}
%
\begin{remark}
  Typically, a \emph{macroscale discretization error} is associated with
  the scale H, and (depending on the sampling/upscaling strategy) a
  \emph{model error} and \emph{sampling error} can be associated with the
  scales $\varepsilon$ and $\delta$ \cite{Maier:2014, Maier:2015}. (We have
  omitted introducing a \emph{microscale discretization error} associated
  with a scale $h$.) In principle, all of these scales have to be
  controlled and chosen appropriately in order to achieve a certain
  accuracy, we refer to \cite{Maier:2015} for a detailed discussion. In
  this publication we focus on the control of the model error, i.\,e. on
  the task of finding a suitable model $\bar A^\delta$. We will thus assume
  that discretization errors are well controlled by choosing $H$ and $h$
  suitably small.
\end{remark}


\subsection{Duality-based error identity}

Suppose that a \emph{quantity of interest} is given by the value $\langle
j,u^\varepsilon\rangle$, where $j\in H^{-1}(\Omega)$ is a linear and
continuous functional and $\langle.\,,.\rangle$ denotes the duality
pairing. Define a \emph{dual problem} to find $z^\varepsilon\in
H^1_0(\Omega)$ s.\,t.
\begin{align}
  \label{eq:4dualproblem}
  \big(A^\varepsilon\nabla\varphi,\nabla z^\varepsilon\big)=\langle j,
  \varphi\rangle \quad\forall\,\varphi\in H^1_0(\Omega).
\end{align}
The dual problem is well-posed and its solution immediately gives rise to 
an \emph{error identity}.

\begin{lemma}[Error identity \cite{Maier:2014,Maier:2015}]
  Let $u^\varepsilon$ be the solution of \eqref{eq:3modprob}, $U$ be the
  solution of \eqref{eq:3probdisc}, and $z^\varepsilon$ be the solution of
  \eqref{eq:4dualproblem}. Then,
  \begin{gather}
    \label{eq:4erroridentity}
    \langle j,u^\varepsilon\rangle  - \langle j,U\rangle
    = \underbrace{\big(f,z^\delta\big) -
    \big(\bar A^\delta\nabla U,\nabla z^\delta\big)}_{=:\,\theta^H}
    +\underbrace{\big(\bar A^\delta\nabla u^\delta,\nabla z^\varepsilon\big) -
    \big(A^\varepsilon\nabla u^\delta,\nabla z^\varepsilon\big)}
    _{=:\,\theta^\delta},
  \end{gather}
  with the following two \emph{error estimators}: $\,\theta^H\,$ representing a
  \emph{residual on the macroscale}, and $\,\theta^\delta\,$ estimating the
  model error. Here, $z^\delta\in H^1_0(\Omega)$ is the solution of a
  corresponding \emph{effective dual problem} \cite{Maier:2014},
  \begin{align}
    \label{eq:4dualproblemeff}
    \big(\bar A^\delta\nabla\varphi,\nabla z^\delta\big)=\langle j,\varphi\rangle
    \quad\forall\,\varphi\in H^1_0(\Omega).
  \end{align}
The model error $\theta^\delta$ is splitted into a sum of local
\emph{model-error} indicators:
\begin{align}
  \label{eq:4etadapprox}
  \theta^\delta =\sum_{K\in\,\mathbb{T}_{\delta}(\Omega)}\eta^\delta_K,
  \quad
  \eta^\delta_K \coloneqq \big(\{A^\varepsilon-\bar A^\delta\}\nabla
  u^\delta,\nabla z^\varepsilon\big)_{K}.
\end{align}

\end{lemma}


\subsection{A localization strategy for the dual problem}

A fundamental difficulty arises from the fact that computing the solution
of the dual problem is (in case of the elliptic model problem) of the same
complexity as the primal problem itself. A \emph{global} fine-scale
approximation of $z^\varepsilon$ has to be considered infeasible. Thus, a
strategy to approximate the dual problem with low computational overhead is
needed.

We proposed \cite{Maier:2014, Maier:2015} a strategy that combines the
usage of a global, effective approximation of $z^\varepsilon$ (such as
$z^\delta$) with a local enhancement. The enhancement is given by localized
reconstruction problems in spirit of a \emph{variational multiscale}
ansatz.
%
\begin{definition}[Local enhancement]
  \label{defi:4localenhancement}
  Let $z^\delta$ be the solution of (\ref{eq:4dualproblemeff}),
  and let $\big\{\omega_K\::\:K\in\,\mathbb{T}_{\delta}(\Omega)\big\}$ be a
  set of \emph{reconstruction patches} fulfilling $\omega_K\supset K$.
  Define a patch-wise reconstruction $z^\delta_K\in H^1_0(\omega_K)$ by
  \begin{align}
    \label{eq:4localenhancement}
    \big(A^\varepsilon\nabla\varphi,\nabla(z^\delta+z^\delta_K)\big)
    =\langle j,\varphi\rangle \quad\forall\varphi\in H^1_0(\omega_K).
  \end{align}
\end{definition}

With the choice $\omega(K)=K$, the locally reconstructed dual solution
leads to a conforming ansatz
$z^\delta+\sum_{K\,\in\mathbb{T}_{\delta}}\,z^\delta_K\,\in H^1(\Omega)$.
In this case the above local enhancement strategy can be regarded as a variant
of the VMM formulation that only has a reconstruction coupling from coarse-
to finescale (and omits the opposite compression coupling).

In contrast to residual-type estimators that can be evaluated in a simple
post-processing step, the practical evaluation of the error estimators
require the approximation of an additional, intermediate dual solution and
effective coefficients; for details we refer to
\cite{Maier:2014,Maier:2015}.


\section{Model-optimization framework}
\label{sec:3framework}

In the previous section an error identity and local error estimates were
introduced for the model error, as well as, the macroscale discretization
error. The treatment of discretization errors by adaptive mesh refinement
with the help of local error indicators is well established
\cite{Becker:1996a, Becker:2001}. The question arises what to do in case of
the model error: Based on the concept of the effective model
\begin{align}
  \bar A^\delta:\,\mathbb{T}_{\delta}(\Omega)\to\mathbb{R}^{d\times d},
\end{align}
two fundamentally different approaches for model adaptivity are possible.
The first is based on the refinement of the sampling mesh
$\mathbb{T}_{\delta}(\Omega)$ and associated sampling regions
$\{Y_K^\delta:\,K\in\,\mathbb{T}_{\delta}(\Omega)\}$ while keeping the same
reconstruction process for all sampling regions \cite{Maier:2014}. This is
comparable to a classical discretization adaptation. The second strategy
consists of switching the effective model used for the reconstruction
process~\cite{Oden:2000a,Oden:2000b,Romkes:2007,Braack:2003}. This is done
by locally selecting a more expensive but also more precise sampling
strategy from an a priori chosen list of effective models. Typically, the
same fixed sampling discretization is used throughout the process.

In this section a novel approach for model adaptivity is introduced that
expresses the adaptation process as a \emph{minimization problem} of the
error estimator $\theta^\delta$: Given the \emph{error identity}
(\ref{eq:4erroridentity}),
\begin{align}
  \langle j,u^\varepsilon\rangle - \langle j,U\rangle
  =\theta^H+\theta^\delta,
  \quad
  \theta^\delta=\big( (\bar A^\delta-A^\varepsilon)\nabla u^\delta,
  \nabla z^\varepsilon\big),
\end{align}
model adaptivity is interpreted as solving an \emph{optimization problem}
\begin{align}
  \arginf_{A^\delta}
  \sum_{K\in\mathbb{T}_{\delta}(\Omega)}\Big[
  \big|\big((\bar A^\delta-A^\varepsilon)\nabla u^\delta,
  \nabla z^\varepsilon\big)_K\big|^2+\text{regularization}\Big].
\end{align}
This approach can be used as a \emph{model-optimization framework} to
locally select optimal coefficients from a set of available models, as well
as in situations where a strategy to derive an effective model is not known
and, thus, an efficient post-processing strategy is needed to construct
one. The latter approach has the advantage that no a priori knowledge about
effective models and reconstruction principles has to be available. The
optimization problem itself is used to select the optimal model.
For the sake of simplicity we will omit the macroscale discretization error
$\theta^H$ in the subsequent discussion. We thus choose $H$ to be suitably
small to guarantee $U\approx u^\delta$. Techniques to control $\theta^H$
simultaneously with the model error are discussed in \cite{Maier:2015}.


\subsection{An optimization approach}

The \emph{quality} of an effective model $\bar A^\delta$ \emph{with respect
to a quantity of interest} $\,\langle j,u^\varepsilon\rangle\,$ can be
measured with the help of the error identity (\ref{eq:4erroridentity}):
\begin{align}
  \label{eq:5functional}
  \langle j,u^\varepsilon\rangle-\langle j,u^\delta\rangle
  =
  \big( (A^\delta-A^\varepsilon)\nabla u^\delta(\bar A^\delta),
  \nabla z^\varepsilon\big)_{L^2(\Omega)^d}.
\end{align}
Given a fixed, a priori chosen sampling discretization
$\mathbb{T}_{\delta}(\Omega)$, define a set of \emph{admissible
coefficients} consisting of symmetric and elliptic coefficient tensors 
(as defined in \eqref{eq:2elliptic},
\begin{align}
  \mathcal{A}^\delta\coloneqq\big\{
  \bar A^\delta:\,\mathbb{T}_{\delta}(\Omega)\to\mathbb{R}^{d\times d}\,:\,
  \bar A^\delta\text{ fulfills (\ref{eq:2elliptic})}\big\}.
\end{align}
%

\begin{definition}[Model-optimization problem]
  \label{defi:5modelopt}
  Let $\bar A^{\delta,0}$ be an initial effective model and let
  $\{\alpha_K\}_{K\in\mathbb{T}_{\delta}(\Omega)}$,
  $\alpha_K\in\mathbb{R}^+$, be a set of (local) regularization parameters.
  Then, an optimal model $\bar A^{\delta,\text{opt}}$ is defined to be a
  solution of
  \begin{align}
    \label{eq:5optprob}
    \arginf_{\bar A^\delta\in\mathcal{A}^\delta}
    \sum_{K\in\,\mathbb{T}_{\delta}(\Omega)} \left\{
    \big|\big( (\bar
    A^\delta-A^\varepsilon)\nabla u^\delta(\bar A^\delta),\nabla
    z^\varepsilon\big)_K\big|^2 + \alpha_K\big\|\,\bar A^\delta_K-\bar
    A^{\delta,0}_K\big\|^2_{\mathbb{R}^{d\times d}} \right\},
  \end{align}
  subject to the side condition
  \begin{align}
    \label{eq:5sidecond}
    \big(\bar A^\delta\nabla u^\delta(\bar
    A^\delta),\nabla\varphi\big)=(f,\varphi)\quad\forall\varphi\in
    H^1_0(\Omega).
  \end{align}
\end{definition}

{\noindent}The regularization parameters $\alpha_K$ are best fixed to 
a uniform value
$\alpha_K=\alpha_0$ on all sampling regions. Here, $\alpha_0$ is chosen to
be roughly $0.01-1$ times the typical size of
$|\tilde\theta^\delta|^2/|\bar A^\delta_K|^2$.
%
\begin{remark}
  \label{rem:5ellipticity}
  Given the fact that ellipticity (\ref{eq:2elliptic}) is impractical to
  enforce, because the correct lower-bound $\,\alpha\,$ is usually not known,
  the ellipticity constraint present in $\mathcal{A}^\delta$ is dropped in
  the concrete numerical computations. The regularization together
  with a factor $\,\alpha_K\,$ appropriately chosen is enough to ensure
  sensible coefficients $\,\bar A^\delta$.
\end{remark}
%
\begin{proposition}
  \label{prop:5wellposedoptprob}
  The optimization problem admits a (not necessarily unique) minimum.
\end{proposition}
%
\begin{proof}
  The functional dependency $\,u^\delta(\bar A^\delta)\,$ described by
  (\ref{eq:5sidecond}) with respect to $\,\bar A^\delta\in\mathcal{A}^\delta\,$
  is well-posed---i.\,e., (\ref{eq:5sidecond}) is always uniquely
  solvable---and continuous. Further,
  \begin{align}
    \|\nabla u^\delta(\bar A^\delta)\|\le\frac1\alpha\|f\|,
  \end{align}
  by definition of $\mathcal{A}^\delta$. Hence, the function
  \begin{align}
    \label{eq:5mathcalf}
    \mathcal{F}(\bar A^\delta)\coloneqq
      \sum_{K\in\,\mathbb{T}_{\delta}(\Omega)}\;\left\{
      \big|\big( (\bar A^\delta-A^\varepsilon)\nabla u^\delta(\bar A^\delta),
      \nabla z^\varepsilon\big)_K\big|^2
      + \alpha_K\big\|\,\bar A^\delta_K-A^{\delta,0}_K\big\|^2
      _{\mathbb{R}^{d\times d}}\right\}
  \end{align}
  is well-defined, continuous, and coercive, i.\,e., it holds true that
  \begin{align}
    \mathcal{F}(\bar A^\delta)\to\infty\;\;\;\text{for}\;\;\;\|\,\bar
    A^\delta\|\to\infty.
  \end{align}
  The optimization problem thus possesses a minimizer.
\end{proof}
%
\begin{remark}
  The functional dependency $\,u^\delta(\bar A^\delta)\,$ given by the
  side-condition (\ref{eq:5sidecond}) is highly nonlinear. In fact,
  $\,\|\nabla u^\delta\|_{L^2(K)}\to 0\,$ has to be expected for the limit
  $\,\|\,\bar A^\delta_K\|\to\infty$. Consequently, the term $\,\big|\big(
  (\bar A^\delta-A^\varepsilon)\nabla u^\delta(\bar A^\delta),\nabla
  z^\varepsilon\big)_K\big|^2\,$ is generally not convex. The optimization
  problem is therefore not uniquely solvable in general.
\end{remark}
 
In preparation for the numerical treatment of the optimization problem
(\ref{eq:5optprob}), we formulate the following regularity result for the
cost functional $\,\mathcal{F}\,$ given in (\ref{eq:5mathcalf}).
%
\begin{proposition}
  \label{prop:5gateaux}
  The functional dependency $\,\mathcal{F}(\bar A^\delta)\,$ is
  \emph{Gâteaux-differentiable} and its derivative
  $\,\mathrm{D}\mathcal{F}(\bar A^\delta)[\delta\bar
  A^\delta]\,$ in direction $\,\delta\bar A^\delta\,$ is
  given by
  \begin{multline}
    \mathrm{D}\mathcal{F}(\bar A^\delta)[\delta\bar A^\delta]
    = \sum_{K\in\,\mathbb{T}_{\delta}(\Omega)} \Big\{
    2\eta^\delta_K\,\big(\delta\bar A^\delta_K\nabla u^\delta(\bar
      A^\delta)\,,\nabla z^\varepsilon\big)_K
    \\
    \qquad+\; 2\eta^\delta_K\, \big((\bar A^\delta-A^\varepsilon)
    \nabla w\,,\nabla z^\varepsilon\big)_K
    +\; 2\alpha_K\,(\bar
    A^\delta_K-\bar A^{\delta,0}_K):\delta\bar A^\delta_K \Big\},
  \end{multline}
  with the solution $\,w\,$ $\big(=\mathrm{D}u^\delta(\bar
  A^\delta)[\delta\bar A^\delta]\,\big)$ of the equation
  \begin{align}
    \label{eq:5microscaleresponse}
    \big(\bar A^\delta\nabla w\,,\nabla\varphi\big)
    +\big(\delta\bar A^\delta\nabla u^\delta(\bar A^\delta),
    \nabla\varphi\big)
    =0
    \quad\forall\varphi\in H^1(\Omega).
  \end{align}
\end{proposition}
%
\begin{proof}
  The crucial part is to assert that the side condition
  (\ref{eq:5sidecond}) interpreted as a functional dependency
  $\,u^\delta(\bar A^\delta)\,$ is Gâteaux-differentiable and its derivative is
  given by (\ref{eq:5microscaleresponse}). The rest of the statement then
  follows in a straightforward manner.
  Due to the fact that $\mathcal{A}^\delta$ is finite dimensional it
  suffices to show that the limit
  \begin{align}
    \label{eq:5limit}
    \lim_{s\searrow0}\,w_s, \quad w_s
    :=\frac1s\big( u^\delta(\bar A^\delta+\delta\bar A^\delta)
    -u^\delta(\bar A^\delta)\big)
  \end{align}
  is well-defined for arbitrary $\delta\bar A^\delta$. For this, we note
  that for $\,s\,$ sufficiently small, the difference $\,w_s\,$ is given by
  \begin{align}
    \big((\bar A^\delta+s\delta\bar A^\delta)\nabla(u^\delta(\bar
    A^\delta)+s w_s)\,,\nabla \varphi\big)
    =(f,\varphi)\quad\forall\varphi\in H^1(\Omega).
  \end{align}
  Equivalently,
  \begin{align}
    \label{eq:5limit2}
    \big(\bar A^\delta\nabla w_s\,,\nabla\varphi\big)
    +s\,\big(\delta\bar A^\delta\nabla w_s\,,\nabla\varphi\big)
    +\big(\delta\bar A^\delta u^\delta(\delta A^\delta)\,,\nabla\varphi\big)
    =0.
  \end{align}
  By continuity, it follows that the limit of (\ref{eq:5limit2}) for
  $\,s\to 0\,$ is well-defined and indeed given by
  (\ref{eq:5microscaleresponse}).
\end{proof}

\subsection{An efficient post-processing strategy}
\label{sec:5optprobpostprocessing}

The optimization problem (\ref{eq:5optprob}) can be used as an efficient
post-processing strategy that does not require---with the exception of an
initial model---any additional a priori knowledge of effective models. Fix
a macroscale and a sampling discretization $\mathbb{T}_{H}(\Omega)$ and
$\mathbb{T}_{\delta}(\Omega)$, as well as an initial effective coefficients
$\bar A^{\delta,0}$. This results in the following general
optimization strategy formulated for the case of a reduced, locally
enhanced approximation of the dual solution:

\begin{definition}[Reduced, locally enhanced model-optimization problem]
  \label{defi:5optprobmod}
  \mbox{}\\
  Let $\mathbb{T}_{\delta}(\Omega)$ be a fixed sampling mesh and
  $\mathbb{T}_{H}(\Omega)$ a fixed macroscale discretization. Fix a
  microscale discretization
  $\{\mathbb{T}_{h}(K)\,:\,K\in\mathbb{T}_{\delta}(\Omega)\}$ as well and
  let $\bar
  A^{\delta,0}:\mathbb{T}_{\delta}(\Omega)\to\mathbb{R}^{d\times
  d}$ be an initial effective model. The \emph{reduced, locally enhanced
  model-optimization problem} reads: Find a solution $\bar
  A^{\delta,\text{opt}}\in\mathcal{A}^\delta$ of
  \begin{multline}
    \label{eq:5optprobmod}
    \arginf_{\bar A^\delta\in\mathcal{A}^\delta}
    \sum_{K\in\,\mathbb{T}_{\delta}(\Omega)}\Big\{ \big|\big( (\bar
    A^\delta-A^\varepsilon)\nabla U(\bar A^\delta), \nabla\big(\tilde
    Z(\bar A^\delta)+ \tilde Z_K(\bar A^\delta)\big)\big)_K\big|^2
    \\
    + \alpha_K\big\|\,\bar A^\delta_K-\bar
    A^{\delta,0}_K\big\|^2_{\mathbb{R}^{d\times d}}\Big\},
  \end{multline}
  with $\,U,\,\tilde Z\in V^H(\Omega)$, and $\,\tilde Z_K\in V^h(K)\,$ 
  subject to the side conditions:
  \begin{align}
    \big(\bar A^\delta\nabla U(\bar A^\delta),\nabla\varphi\big)
    &=(f,\varphi)\quad\forall\varphi\,\in V^H(\Omega),
    \\[0.3em]
    \big(\bar A^\delta\nabla\varphi,\nabla\tilde Z(\bar A^\delta)\big)
    &=\langle j,\varphi\rangle \quad\forall\,\varphi\in V^H(\Omega),
    \label{eq:5localrec}
    \\[0.3em]
    \big(A^\varepsilon\nabla\varphi,\nabla\tilde Z(\bar A^\delta)
    +\nabla\tilde Z_K(\bar A^\delta)\big)_K
    &= \langle j,\varphi\rangle  \quad\forall\,\varphi\in V^h(K).
    \label{eq:5localenh}
  \end{align}
\end{definition}

Analogously to Proposition~\ref{prop:5gateaux} we formulate the following
result:
%
\begin{proposition}
  \label{prop:5gateaux2}
  Let $\,\tilde{\mathcal{F}}\,$ be the modified cost functional of
  (\ref{eq:5optprobmod}). Then, in full analogy of the result for
  $\,\mathcal{F}\,$ in Proposition~\ref{prop:5gateaux}, the functional
  dependency of $\,\tilde{\mathcal{F}}(\bar A^\delta)\,$ is also
  Gâteaux-differentiable and it holds true that
  \begin{multline}
    \mathrm{D}\tilde{\mathcal{F}}(\bar A^\delta)[\delta\bar A^\delta]=
    \mathrm{D}\mathcal{F}(\bar A^\delta,U,\tilde Z)[\delta\bar A^\delta]\;+
    \\
    \sum_{K\in\,\mathbb{T}_{\delta}(\Omega)} 2\eta^\delta_K\,\big((\bar
    A^\delta-A^\varepsilon)\nabla U,\nabla(\mathrm{D}\tilde
    Z+\mathrm{D}\tilde Z_K) (\bar A^\delta)[\delta\bar A^\delta]\big),
  \end{multline}
  with $\,\mathrm{D}\tilde Z\in V^H(K)\,$ being defined as the solution of
  \begin{align}
    \label{eq:5microscaleresponseZ}
    \big(\bar A^\delta\nabla\varphi,\nabla\mathrm{D}\tilde Z(\bar
    A^\delta)[\delta\bar A^\delta]\big) +\big(\delta\bar
    A^\delta\nabla\varphi,\nabla\tilde Z(\bar A^\delta)\big)
    =0
    \quad\forall\varphi\in V^H(\Omega),
  \end{align}
  and $\,\mathrm{D}\tilde Z_K\in V^h(K)\,$ solving
  \begin{multline}
    \label{eq:5microscaleresponseZK}
    \big(A^\varepsilon\nabla\varphi,\nabla\mathrm{D}\tilde Z_K(\bar
    A^\delta)[\delta\bar A^\delta]\big)_K
    +\big(A^\varepsilon\nabla\varphi,\nabla\delta\tilde Z(\bar
    A^\delta)[\delta\bar A^\delta]\big)_K
    =0
    \quad\forall\varphi\in V^h(K).
  \end{multline}
\end{proposition}
%
\begin{proof}
  The first part of the statement is already proved in
  Proposition~\ref{prop:5gateaux}. The additional terms arise from the
  derivatives of Equations (\ref{eq:5localrec}) and (\ref{eq:5localenh}).
\end{proof}


\section{Implementational aspects}
\label{sec:4implementational}

The optimization problem (\ref{eq:5optprob}) and its modified variant
(\ref{eq:5optprobmod}) contain strongly nonlinear side conditions, where
computing the Gâteaux-derivative for a given direction $\delta\bar
A^\delta$ alone already involves solving the variational equation
(\ref{eq:5microscaleresponse}), and, depending on the reconstruction
approach, also (\ref{eq:5microscaleresponseZ}) and
(\ref{eq:5microscaleresponseZK}). Consequently, a straightforward
application of the \emph{Newton method} to solve the optimization problem
has to be avoided.


\subsection{Gauss-Newton Method}

In order to avoid computing the second order derivatives
$\mathrm{d}^2\mathcal{F}(\bar A^{\delta,i})$ a \emph{modified Gauss-Newton
method} \cite{Marquardt:1963} is used. For this, we
reformulate the optimization problem slightly. Introduce a multi-index
$(K,i,j)\in\mathbb{T}_{\delta}\times\mathbb{R}^{d\times d}$ and define the
vector-valued function
\begin{gather}
  \mathcal{G} \coloneqq
  \big\{\big(\eta_K\big)_K\,,\,\big(g_{Kij}\big)_{Kij}\big\},
  \quad\text{with}
  \\
  \eta_K \coloneqq
  \big((\bar A^\delta-A^\varepsilon)\nabla U(\bar
  A^\delta)\,,\nabla(\tilde Z+\tilde Z_K)\big)_K,
  \quad
  g_{Kij} \coloneqq \sqrt{\alpha_K}\,\big(\bar
  A^\delta_{K,ij}-\bar{A}^{\delta,0}_{K,ij}\big).
\end{gather}
%
\begin{lemma}
  The modified optimization problem (\ref{eq:5optprobmod}) can equivalently
  be expressed as the minimization of the squared \emph{Euclidian norm}
  $\,|\cdot|\,$ of $\mathcal{G}$:
  \begin{align}
    \arginf_{\bar A^\delta\in\mathcal{A}^\delta}
    \big|\mathcal{G}\,\big|^2
    \:=\:
    \arginf_{\bar A^\delta\in\mathcal{A}^\delta}
    \sum_{K\in\,\mathbb{T}_{\delta}(\Omega)}
    \Big\{
    \eta_K^2+
    \sum_{ij}g_{Kij}^2
    \Big\}.
  \end{align}
\end{lemma}
%
For a given index $\,(K,i,j)\,$ let $\,\delta\bar
A^\delta(Kij)\,:\,\mathbb{T}_{\delta}\to\mathbb{R}^{d\times d}\,$ be defined
as the value
\begin{align}
  \big(\delta\bar A^\delta_Q\big)_{mn}:=\delta_{QK}\,
  \delta_{mi}\,\delta_{ni},
\end{align}
for a cell $\,Q\in\mathbb{T}_{\delta}(\Omega)$, where $\,\delta_{QK}\,$ is
the \emph{Kronecker delta}. Define the short notation
\begin{align}
  \text{D}_{Kij}\eta_Q &\coloneqq \mathrm{D}\eta_Q[\delta\bar
  A^\delta(Kij)],
  \\
  \mathrm{D}_{Kij}g_{Qmn} &\coloneqq \mathrm{D} g_{Qmn}(\bar A^\delta)
  [\delta\bar A^\delta(Kij)].
\end{align}
%
\begin{lemma}
  By virtue of Propositions~\ref{prop:5gateaux} and \ref{prop:5gateaux2} it
  holds that
  \begin{multline}
    \label{eq:5derivative}
    \mathrm{D}_{Kij}\eta_Q \;=\; \delta_{QK}
    \int_Q\nabla_{j}U\,\nabla_{i}(\tilde Z+\tilde Z_Q)\,\mathrm{d} x
    \\
    + \int_Q(\bar A^\delta-A^\varepsilon)\nabla\mathrm{D}_{Kij}U\cdot
    \nabla(\tilde Z+\tilde Z_Q)\,\mathrm{d} x
    \\
    + \int_Q(\bar A^\delta-A^\varepsilon)\nabla
    U\cdot\nabla\mathrm{D}_{Kij} (\tilde Z+\tilde Z_Q)\,\mathrm{d} x,
  \end{multline}
  as well as
  \begin{align}
    \mathrm{D}_{Kij}g_{Qmn}= \delta_{QK}\, \delta_{mi}\,
    \delta_{ni}\,\sqrt{\alpha_Q}.
  \end{align}
\end{lemma}
%
With these prerequisites at hand, a modified Gauss-Newton iteration
following a discussion by Levenberg and Marquardt~\cite{Marquardt:1963} is
defined:
%
\begin{definition}[Gauss-Newton iteration]
  Let $\mathcal{J}$ be the Jacobian matrix of $\mathcal{G}$,
  \begin{align}
    \mathcal{J}=\Big\{
    \big(\mathrm{D}_{Kij}\eta_Q\big)^{Kij}_Q\,,\,
    \big(\mathrm{D}_{Kij}\tilde g_{Qmn}\big)^{Kij}_{Qmn}
    \Big\}.
  \end{align}
  Given a penalty $\,\lambda\ge0\,$ and starting from an initial effective
  model $\,\bar A^{\delta,0}\,$ the modified Gauss-Newton
  iteration reads
  \begin{align}
    \label{eq:5levenberg}
    \begin{cases}
      \bar A^{\delta,n+1}\gets\bar A^{\delta,n}
      +\delta\bar A^{\delta,n},
      \\[0.5em]
      \big(\mathcal{J}\mathcal{J}^T
      (\bar A^{\delta,n}) + \lambda\, \text{Id}\,\big)
      \delta\bar A^{\delta,n} =
      -\mathcal{J}\mathcal{G}^T(\bar A^{\delta,n}).
    \end{cases}
  \end{align}
\end{definition}
%
The penalization term $\,\lambda\,\text{Id}\,$ acts as a damping term in the
Gauss-Newton method to stabilize the iteration and to reduce the influence
of approximation errors of the Jacobian $\mathcal{J}$. Depending on the
situation, it will be chosen between $0-1$ times the mean value of the
diagonal elements of $\,\mathcal{J}\mathcal{J}^T$. In general, the number
of Gauss-Newton iterations will depend on the problem and sampling
strategy, as well as the strength of the regularization.


\subsection{Reduction of computational complexity}

The computationally expensive part of computing the Jacobi matrix
$\,\mathcal{J}\,$ are the non-local responses $\,\mathrm{D}_{Kij}U$,
$\,\mathrm{D}_{Kij}\tilde Z$, and $\,\mathrm{D}_{Kij}\tilde Z_Q\,$ that 
have to be computed for each choice $\,(K,i,j)\,$ individually according to
(\ref{eq:5microscaleresponse}), (\ref{eq:5microscaleresponseZ}), and
(\ref{eq:5microscaleresponseZK}). Another aspect that has to be kept in
mind is the fact that $\,\mathcal{J}\mathcal{J}^T\,$ is actually a dense matrix
of size $\,N\times N\,$ with
$\,N=\big|\mathbb{T}_{\delta}(\Omega)\big|\,(1+d^2)$. Storing such a matrix,
even for moderate sizes of the sampling mesh $\,\mathbb{T}_{\delta}(\Omega)$,
is computationally infeasible. Thus, a reduction strategy to efficiently
approximate $\,\mathcal{J}\,$ is necessary.

The microscale response $\,\mathrm{D}_{Kij}U\,$ is given by
(cf.~Equation~\ref{eq:5microscaleresponse}):
\begin{align}
  \big(\bar A^\delta\nabla\mathrm{D}_{Kij}U\,,\nabla\varphi\big)
  =
  -\int_K\nabla_{ j}U\nabla_{ i}\varphi\,\mathrm{d} x
  \quad\forall\varphi\in V^H(\Omega).
\end{align}
The right hand side of this equation is highly localized. Consequently, the
contribution of
\begin{align}
    \int_Q(\bar A^\delta-A^\varepsilon)\nabla\mathrm{D}_{Kij}U\cdot
    \nabla(\tilde Z+\tilde Z_Q)\,\mathrm{d} x
\end{align}
rapidly decreases the farther $\,Q\,$ is away from $\,K\,$---and can be neglected
at some point. A sensible compromise is, for example, to compute the above
contribution only for the case $\,K=Q$, or alternatively, as a more precise
strategy, only if $\,Q\,$ belongs to a small patch around $\,K$, e.\,g., if
$\,\overline K\cap\overline Q\neq\emptyset$. All of these choices result in a
block diagonal matrix $\,\tilde{\mathcal{J}}\,$ whose \emph{band size} is
independent of $\,|\mathbb{T}_{\delta}(\Omega)|$.

In contrast the microscale response $\mathrm{D}\tilde Z$, $\mathrm{D}\tilde Z_Q$ 
will just be neglected entirely:
\begin{align}
  \label{eq:5neglected}
  \int_Q(\bar A^\delta-A^\varepsilon)\nabla U\cdot\nabla\mathrm{D}_{Kij}
  (\tilde Z+\tilde Z_Q)\,\mathrm{d} x \;\;\approx\;\;0.
\end{align}
The reasoning behind this choice is the fact that in case of a fully
resolved dual solution $z^\varepsilon$, such finescale response does not
exist at all. To further justify this approach, we will discuss detailed
numerical comparisons of optimization results obtained by $z^\varepsilon$
and $z^\delta+\sum z^\delta_K$, respectively, in
Section~\ref{sec:5numerical}.

In summary, the following approximation strategies of the derivative
$\,\mathrm{D}_{Kij}\eta_Q\,$ will be considered:
%
\begin{definition}[Approximative Jacobian]
  Define a patch
  $\omega(K):=\{Q\in\mathbb{T}_{\delta}(\Omega)\,:\,\overline
  K\cap\overline Q\neq\emptyset\}$ and let $\,I_{ Q\omega(K)}\,$ be the 
  indicator function that is equal to $\,1\,$ for $\,Q\in\omega(K)\,$ 
  and $\,0\,$ otherwise. The
  derivative $\,\mathrm{D}_{Kij}\eta_Q\,$ is  approximated by
  \begin{multline}
    \label{eq:5approx3}
    \mathrm{D}_{Kij}\eta_Q \;\approx\; \delta_{ QK}
    \int_Q\nabla_{ j}U\,\nabla_{ i}(\tilde Z+\tilde Z_Q)\,\mathrm{d} x
    \\
    + I_{ Q\omega(K)}
    \int_Q(\bar A^\delta-A^\varepsilon)\nabla\mathrm{D}_{Kij}U\cdot
    \nabla(\tilde Z+\tilde Z_Q)\,\mathrm{d} x.
  \end{multline}
\end{definition}
%
One last obstacle for the patch-centered reconstruction (\ref{eq:5approx3})
remains. Namely, that the response $\,D_{Kij}U\,$ is needed in combination with
the microscale reconstruction $\,\tilde Z_Q\,$ for different $K$ and $\,Q$. In an
efficient algorithm, fine-scale reconstructions of such kind cannot be
stored for further use. They have to be kept local to the computation on
the current sampling region $\,Q$. One way to mitigate this problem is to not
use a finescale reconstruction $Z_Q$ defined on $\,Q$, but to use a slightly
more expensive $\,\tilde Z_{\omega(K)}\,$ defined on the patch $\,\omega(K)\,$
around $\,K\,$ with patch-depth $\,1$. This allows for an efficient assembly as
described in Algorithm~\ref{alg:5assembly}. Finally, a model-optimization
algorithm can be defined; see Algorithm~\ref{alg:5modeloptimization}.
%
\begin{algorithm}[t]
  \DontPrintSemicolon
  -- Set up $\mathbb{T}_{H}(\Omega)$ and assemble matrix $A$:
  $A_{\nu\mu}=(\bar A^{\delta,i}\nabla\varphi_\mu\,,\nabla\varphi_\nu)$\;
  -- Compute matrix decomposition of $A$: $LU=A$\;
  \For{$K\in\mathbb{T}_{\delta}(\Omega)$}{
    -- Assemble $\omega(K)$ and compute $\tilde Z_K\in V^h(\omega(K))$\;
    -- Compute $\eta^\delta_K$ with (\ref{eq:4etadapprox})

    \For{$i=1,\,\ldots,\,d$}{
      \For{$j=1,\,\ldots,\,d$}{
        -- Compute response $\mathrm{D}_{Kij}U$ with above decomposition\;
        \For{$Q\in\omega(K)$}{
          -- Compute contribution $\mathrm{D}_{Kij}\eta_Q$ for
          $\mathcal{J}$ according to (\ref{eq:5approx3})\;
        }
      }
    }
  }
  \caption{Assembly of $\{\eta_K\}$ and $\mathcal{J}$}
  \label{alg:5assembly}
\end{algorithm}
%
\begin{algorithm}[t]
  \DontPrintSemicolon
  -- Compute initial model $\bar A^{\delta,0}$\;
  \vspace{0.2em}
  -- Solve primal and dual problem for $U(\bar A^{\delta,0})$,
  $\tilde Z(\bar A^{\delta,0})$ with the help of (\ref{eq:3probdisc})\;
  \vspace{0.5em}
  \While{stopping criterion not reached}{
    \vspace{0.5em}
    -- Compute the error estimator and local indicators $\{\eta_K\}$
    \begin{align*}
      \tilde\theta^\delta=\sum_{K\in\,\mathbb{T}_{\delta}(\Omega)}
      \eta^\delta_K,
    \end{align*}
    as well as, the Jacobian $\mathcal{J}$ with
    Algorithm~\ref{alg:5assembly}\;
    \vspace{0.5em}
    -- Solve
      $\big(\mathcal{J}\mathcal{J}^T(\bar A^{\delta,n}) +
      \lambda \text{Id}\,\big) \delta\bar A^{\delta,n} =
      -\mathcal{J}\mathcal{G}^T(\bar A^{\delta,n}).$\;
    \vspace{0.5em}
    -- Update model:
    $\bar A^{\delta,n+1}\gets\bar A^{\delta,n}
    +\delta\bar A^{\delta,n}$\;
    \vspace{0.5em}
    -- Compute $U(\bar A^{\delta,n+1})$, $\tilde Z(\bar A^{\delta,n+1})$
    again with updated model $\bar A^{\delta,n+1}$, (\ref{eq:3probdisc})\;
  }
  \caption{Model-optimization algorithm}
  \label{alg:5modeloptimization}
\end{algorithm}
%
\begin{remark}
  Due to the fact that $\mathcal{J}$ is always approximated with a
  substantially reduced variant, the value $\|\mathcal{J}\|$ does not
  provide a good stopping criterion with $\|\mathcal{J}\|\ll1$. Instead, it
  is better to use the approximative estimator value
  $|\tilde\theta^\delta|$ directly. For example, stop if
  $|\tilde\theta^\delta|$ is reduced to $1\,\%$ of its initial value.
\end{remark}


\section{Numerical tests}
\label{sec:5numerical}

A series of short numerical tests is conducted in order to examine specific
behavior and aspects of the model-optimization approach that was proposed
in the previous sections. The computations are done with the finite-element
toolkit deal.II \cite{Bangerth:2016}.

In particular, the dependence of the optimization result on the initial
value $\;\bar A^{\delta,0}$, on the size of the sampling discretization
$\;\mathbb{T}_{\delta}(\Omega)$, and on the strength of the regularization
parameters $\;\alpha_K\;$ is examined for a global functional, as well as a
local variant (Subsection~\ref{subse:5modelrandom}). This is done with a
choice of random coefficients for both, the fully resolved dual solution
$\;z^\varepsilon$, as well as the reduced, locally enhanced variant
$\;z^\delta+\sum z^\delta_K$.  Finally, Subsection~\ref{sec:5transport}
concludes with an advection-diffusion example with dominant transport.


\subsection{Parameter study for random coefficients}
\label{subse:5modelrandom}

\begin{figure}
  \footnotesize
  \centering
  \begin{minipage}[b]{0.4\textwidth}
    \centering
    (a)\;
    \includegraphics[height=3.5cm] {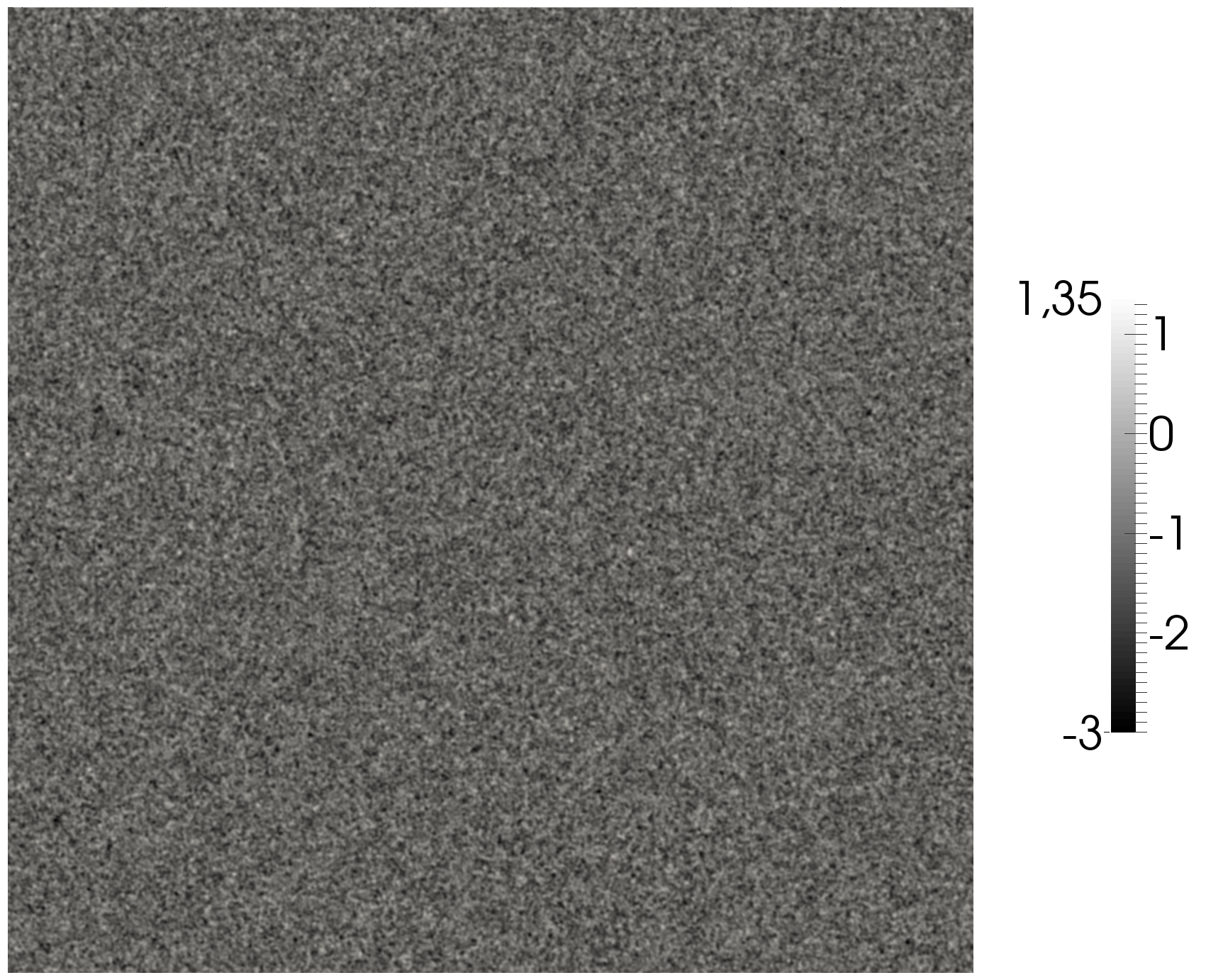}
  \end{minipage}
  \begin{minipage}[b]{0.4\textwidth}
    \centering
    (b)\;
    \includegraphics[height=3.5cm] {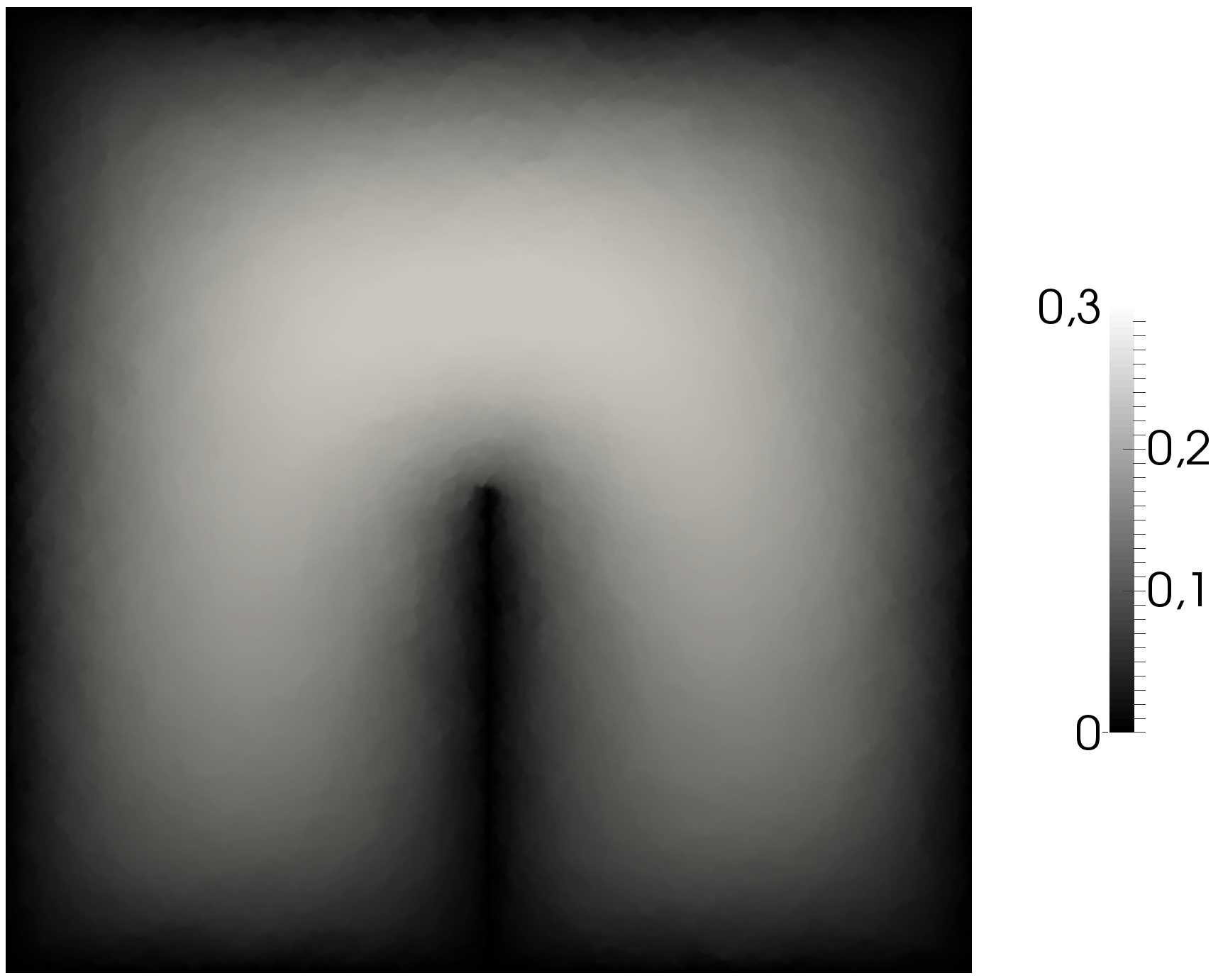}
  \end{minipage}

  \caption{%
  (a) Log-normally distributed permeability with Gaussian correlation shown
  in log-scale. (b) The corresponding reference solution.}
  \label{fig:5quantim}
\end{figure}

The purpose of the first numerical test is to examine the stability of the
optimization approach for a variety of differently chosen discretization
and optimization parameters. In particular, the feasibility of using the
reduced, locally enhanced approximation approach within the optimization
framework shall be assessed as this property is essential for the
optimization approach to be computationally feasible and thus comparable to
VMM or HMM approaches.

Consider a computational domain $\;\Omega=(0,1)^2\;$ with a log-normally
distributed, random microstructure. In detail, we choose $\;A^\varepsilon\;$ 
to be
\begin{align}
   A^\varepsilon(x) := I_d\times\gamma\times\exp(10\times g(x)\,/\,255)\,I_d\,,
\end{align}
where $\,g(x)\,$ is an 8 bit grayscale picture (with integral values
between $\;0\;$ and $\,255$) with $\,1024\!\times\!1024\,$ pixels resolution
(cf.~Fig.~\ref{fig:5quantim}). The grayscale picture is generated using the
{\tt QuantIm} library \cite{Vogel:2008}. It is a (discrete) Gaussian random
field with an additional Gaussian correlation with a correlation length
chosen to be $\,r=0.0025$.

Further, define a global and a local functional as follows:
\begin{align}
  \label{eq:functionals}
  \langle j_1,\varphi\rangle=\int_\Omega\varphi\,\mathrm{d} x,
  \quad
  \langle j_2,\varphi\rangle=\varphi(x_0).
\end{align}
For a fixed choice of $\;65.5\,K\;$ macrocells and a microscale resolution 
of $\,h=2^{-12}$, a parameter study is conducted with sampling discretizations
$\,\delta=2^{-3}\;$ and $\,\delta=2^{-4}$, a choice of mild penalty with
$\,\lambda=0.1m\,$ and regularization $\,\alpha_K=0.001\,$ and strong 
penalty $\,\lambda=1.0m\,$ and regularization $\,\alpha_K=0.01$, where 
$\,m\,$ is the absolute mean value of the diagonal entries of the matrix
$\,\mathcal{J}^T\mathcal{J}$, see (\ref{eq:5levenberg}). The optimization
algorithm is run for the optimization strategy with precise approximative
Jacobian (\ref{eq:5approx3}) for both types of reconstruction approaches
for the dual solution: fully resolved $\,z^\varepsilon\,$ and the reduced,
local enhanced variant $\,z^\delta+\sum z^\delta_K$. With reference values of
$\,\langle j_1,u^\varepsilon_{\text{ref}}\rangle\approx0.14641\,$ and $\,\langle
j_2,u^\varepsilon_{\text{ref}}\rangle\approx0.189403\,$ the initial model
errors are in the range of around $\,1\,\%\,$ for the geometric average.

For each choice of parameters, Table~\ref{tab:5modelproblemj1} shows the
final error after a fixed number of $\,15\,$ optimization cycles for periodic and
random coefficients. The first observation that can be made is that in all
cases the model-optimization approach is able to consistently reduce well
below $\,1\,\%$. More importantly, the reduced, locally enhanced variant
$\,z^\delta+\sum z^\delta_K\,$ with increased patch size (and thus reduced
impact of the artificial Dirichlet boundary conditions of the
reconstruction problems) leads to comparable results very similar to the
results for the full variant $\,z^\varepsilon$.

\begin{table}[tbp]
  \center
  \caption{Parameter study for a random permeability and the global
    functional $j_1$, as well as the local functional $j_2$. For each choice
    the absolute and relative error after cycle 15 of the optimization
    algorithm is shown.}
  \label{tab:5modelproblemj1}
  \medskip

  \subfloat[Full model-optimization (\ref{eq:5approx3}), geometric average
    $\bar A^{\delta,0}$, global functional $j_1$]{
    \footnotesize
    \begin{tabular} {l l c c c c}
      \toprule
      & & \multicolumn{2}{c}{$z^\varepsilon$} & \multicolumn{2}{c}{$z^\delta+\sum z^\delta_K$} \\
      \cmidrule(lr){3-4} \cmidrule(lr){5-6}
      & cycle & $\delta=2^{-3}$ & $\delta=2^{-4}$ & $\delta=2^{-3}$ & $\delta=2^{-4}$ \\
                          & 1  & 1.8e-3 (1.3\,\%) & 1.3e-3 (0.9\,\%) &
                          1.8e-3 (1.3\,\%) & 1.3e-3 (0.9\,\%) \\[0.5em]
      $\alpha_K=10^{-2}$  & 15 & 7.5e-4 (0.5\,\%) & 4.2e-4 (0.3\,\%) & 6.7e-4 (0.5\,\%) & 3.7e-4 (0.3\,\%) \\
      $\alpha_K=10^{-3}$  & 15 & 8.0e-4 (0.6\,\%) & 7.9e-4 (0.5\,\%) & 7.1e-4 (0.5\,\%) & 7.2e-4 (0.5\,\%) \\
      \bottomrule
    \end{tabular}
  }

  \subfloat[Full model-optimization (\ref{eq:5approx3}), geometric average
    $\bar A^{\delta,0}$, local functional $j_2$]{
    \footnotesize
    \begin{tabular} {l l c c c c}
      \toprule
      & & \multicolumn{2}{c}{$z^\varepsilon$} & \multicolumn{2}{c}{$z^\delta+\sum z^\delta_K$} \\
      \cmidrule(lr){3-4} \cmidrule(lr){5-6}
      & cycle & $\delta=2^{-3}$ & $\delta=2^{-4}$ & $\delta=2^{-3}$ & $\delta=2^{-4}$ \\
                          & 1  & 2.1e-3 (1.1\,\%) & 2.2e-3 (1.2\,\%) &
                          2.1e-3 (1.1\,\%) & 2.2e-3 (1.2\,\%) \\[0.5em]
      $\alpha_K=10^{-2}$  & 15 & 6.0e-4 (0.3\,\%) & 6.2e-4 (0.3\,\%) & 8.7e-4 (0.5\,\%) & 9.5e-4 (0.5\,\%) \\
      $\alpha_K=10^{-3}$  & 15 & 6.9e-4 (0.4\,\%) & 7.2e-4 (0.4\,\%) & 9.5e-4 (0.5\,\%) & 1.0e-3 (0.5\,\%) \\
      \bottomrule
    \end{tabular}
  }

\end{table}


\subsection{An advection-diffusion example with dominant transport}
\label{sec:5transport}

As second test case consider an \emph{advection-diffusion} problem
\begin{align}
  \label{eq:5advdiff}
  \gamma\,(\nabla
  u^\varepsilon,\nabla\varphi)+(\boldsymbol{b^\varepsilon}\cdot \nabla
  u^\varepsilon,\varphi)=(f,\varphi) \quad\forall\varphi\in V
\end{align}
driven by a divergence-free vector field $\,\boldsymbol{b^\varepsilon}\in
H^{1,\infty}(\Omega)^d$, i.\,e. $\,\nabla\cdot\boldsymbol{b^\varepsilon}=0\,$
a.\,e. on $\,\Omega\,$ and $\,\boldsymbol{b^\varepsilon}\equiv0\,$ on
$\,\partial\Omega$, together with a positive scaling factor
$\,\gamma\in\mathbb{R}^+$. This time, the multiscale character is given by
$\,\boldsymbol{b^\varepsilon}\,$ that shall consist of small (but strong)
eddies. We again use the {\tt QuantIm} library \cite{Vogel:2008} to
construct a random, divergence-free vector field
$\,\boldsymbol{b^\varepsilon}(\boldsymbol{x})$; see \cite{Maier:2015}.
For a given sampling discretization $\,\mathbb{T}_{\delta}(\Omega)$, we define
an averaged transport coefficient
\begin{align}
  \boldsymbol{b^\delta}\,:\,\mathbb{T}_{\delta}(\Omega)\to\mathbb{R}^d,
  \quad
  \boldsymbol{b^\delta_K}:=\fint_K\boldsymbol{b^\varepsilon}\,\mathrm{d}
  x\quad\text{for }K\in\mathbb{T}_{\delta}(\Omega).
\end{align}

The random advection field influences the macroscopic diffusion in two
ways. Firstly, an averaged macroscopic transport occurs (as described by
$\,\boldsymbol{b^\delta}$). Secondly, and more importantly, the microscopic
eddies lead to influence the macroscopic behavior by means of an additional
\emph{effective diffusivity}. Consequently, let the task be to find
effective (diffusion) coefficients $\,\bar
A^\delta:\,\mathbb{T}_{\delta}(\Omega)\to\mathbb{R}^{d\times d}\,$ such that
the solution $u^\delta$ of the \emph{effective advection-diffusion problem}
\begin{align}
  (\bar A^\delta\nabla
  u^\delta,\nabla\varphi)+(\boldsymbol{b^\delta}\cdot\nabla
  u^\delta,\varphi)=(f,\varphi)\quad\forall\varphi\in V,
\end{align}
is a good approximation of $u^\varepsilon$ in some quantity of interest.

The only significant change in the model-adaptation framework for the above
advection-diffusion problem is the occurrence of additional terms
$(\boldsymbol{b^\varepsilon}\cdot\nabla u^\varepsilon,z^\varepsilon)$ in
the error identity (\ref{eq:4erroridentity}) that now splits into
\begin{multline}
  \label{eq:5erroridentityadvdif2}
  \langle j,u^\varepsilon\rangle  - \langle j,U\rangle
  = \underbrace{\big(f,z^\delta\big) -
  \big(A^\delta\nabla U,\nabla z^\delta\big) -
  \big(\boldsymbol{b^\delta}\cdot\nabla U,z^\delta\big)}_{=:\,\theta^H}
  \\
  +\underbrace{\big(A^\delta\nabla u^\delta,\nabla z^\varepsilon\big) -
    \gamma\,\big(\nabla u^\delta,\nabla z^\varepsilon\big)-
    \big((\boldsymbol{b^\varepsilon}-\boldsymbol{b^\delta})\cdot\nabla
  u^\delta,z^\varepsilon\big)}_{=:\,\theta^\delta}
\end{multline}
This leads to a local model-error indicator
\begin{align}
  \eta^\delta_K\coloneqq \big(\{\gamma\,\text{Id}-A^\delta\}\nabla
  u^\delta,\nabla z^\varepsilon\big)_K- \big(
  (\boldsymbol{b^\varepsilon}-\boldsymbol{b^\delta})\cdot\nabla
  u^\delta,z^\varepsilon\big)_K.
\end{align}
With the above assumptions on $\,\boldsymbol{b^\varepsilon}\,$ the corresponding
dual problem reads
\begin{align}
  \gamma\,(\nabla\varphi,\nabla z^\varepsilon)-(\boldsymbol{b^\varepsilon}\cdot
  \nabla z^\varepsilon,\varphi)=\langle
  j,\varphi\rangle \quad\forall\varphi\in V.
\end{align}

\begin{figure}
  \centering
  \subfloat[Domain $\Omega$] {
    \begin{tikzpicture}[scale=2]
      \draw[white] (-0.35,-0.35) rectangle (1.35,2.35);
      \draw[very thick] (0,0) rectangle (1,2);
      \draw[thick] (-0.075,1) -- (0.075,1);
      \node at (0.5,1) {$\Omega$};
      \node at (1.25,1) {$\Gamma_A$};
      \node at (-0.25,0.5) {$\Gamma_D$};
      \node at (-0.25,1.5) {$\Gamma_C$};
      \node at (0.5,2.23) {$\Gamma_B$};
      \node at (0.5,-0.23) {$\Gamma_E$};
    \end{tikzpicture}
  }
  \subfloat[$\boldsymbol{b^\varepsilon}$, $x$-comp.] {
    \begin{tikzpicture}[scale=2]
      \draw[white] (-0.5,-0.35) rectangle (1.5,2.35);
      \node at (0.5,1) {\includegraphics[height=4.0cm]{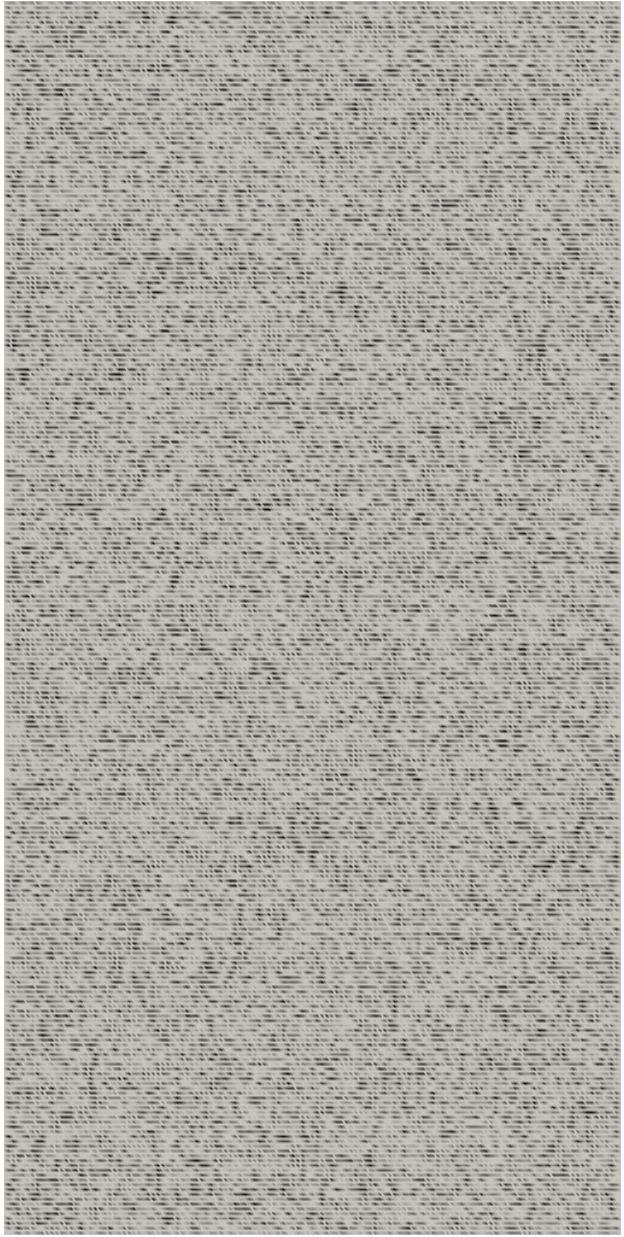}};
    \end{tikzpicture}
  }
  \subfloat[$\boldsymbol{b^\varepsilon}$, $y$-comp.] {
    \begin{tikzpicture}[scale=2]
      \draw[white] (-0.5,-0.35) rectangle (1.5,2.35);
      \node at (0.5,1) {\includegraphics[height=4.0cm]{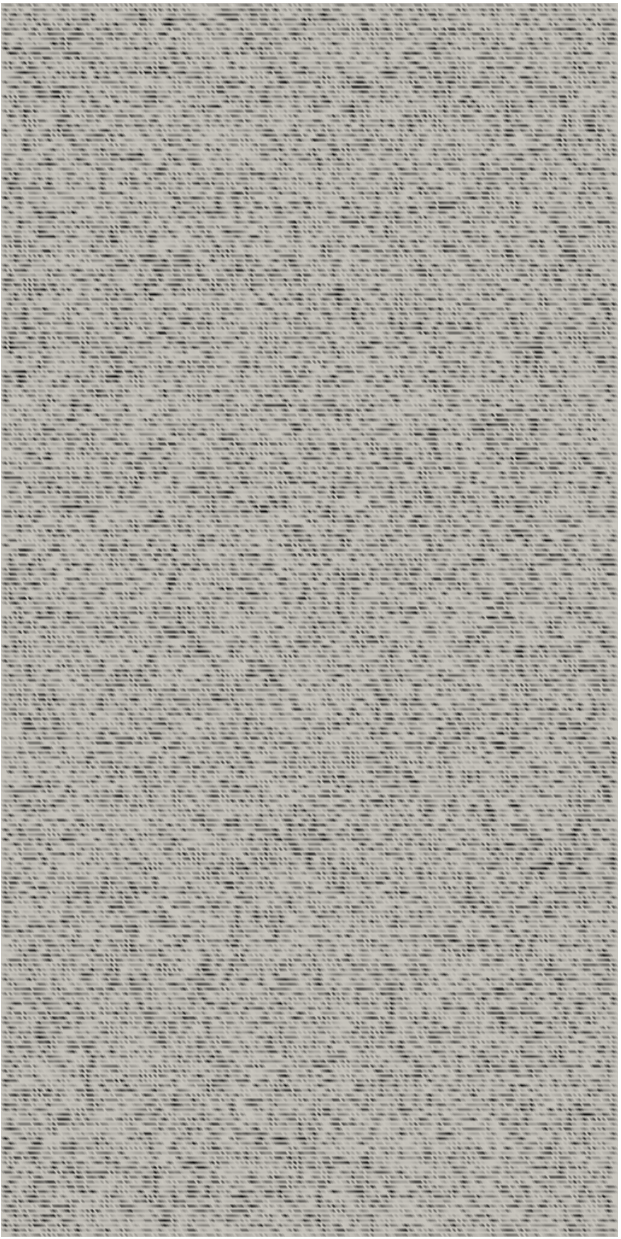}};
    \end{tikzpicture}
  }
  \caption{The computational domain $\Omega$ for the advection-diffusion
    test case (a) and the random vector field $\boldsymbol{b^\varepsilon}$ (b/c).
  }
  \label{fig:5domainadvdiff}
\end{figure}
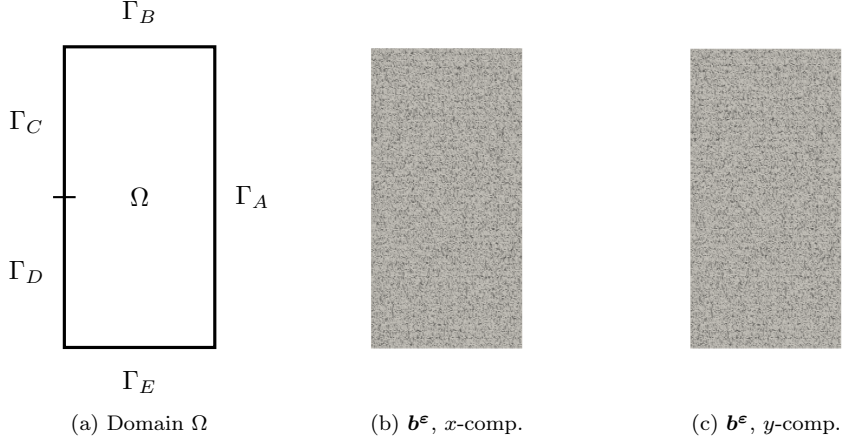

A rectangular domain $\Omega$ is chosen (see
Fig.~\ref{fig:5domainadvdiff}) with homogeneous Dirichlet boundary
conditions on $\Gamma_D$, homogeneous Neumann conditions on $\Gamma_A$,
$\Gamma_B$ and $\Gamma_C$, and
$\,\gamma\,\partial_{\boldsymbol{n}}u^\varepsilon\equiv1\,$ on $\Gamma_E$.
The source term is set to $f\equiv0$ and the quantity of interest is chosen
to be
\begin{align}
  \langle j,\varphi\rangle=\int_{\Gamma_B}\varphi\,\mathrm{d} o_x.
\end{align}

In spirit of Definition~\ref{defi:4localenhancement}, a reduced dual
problem with a local enhancement can be defined
\begin{align}
  \big(\bar A^\delta\nabla\varphi,\nabla
  z^\delta\big)-(\boldsymbol{b^\delta}\cdot\nabla z^\delta,\varphi)
  &=\langle j,\varphi\rangle\quad\forall\varphi\in V,
  \\[0.5em]
  \gamma\,\big(\nabla\varphi,\nabla(z^\delta+z^\delta_K)\big)_K-
  \big((\boldsymbol{b^\varepsilon}-\boldsymbol{b^\delta})\cdot\nabla
  (z^\delta+z^\delta_K),\varphi\big)_K
  &=\langle j,\varphi\rangle\quad\forall\varphi\in V(K).
\end{align}
Here, the local reconstruction $\,z^\delta_K\in V(K)\,$ has homogeneous
Dirichlet conditions on \emph{interior} boundary parts $\,\partial K\,$ but
shall have homogeneous Neumann conditions on all Neumann boundaries of the
primal problem, i.\,e., on boundaries $\,\partial K\cap\Gamma_i\,$ with 
$\,i=A$,
$B$, $C$, $E$. For the choice $\,\varepsilon=2^{-8}$, $\,\gamma=0.1$, as well
as values of the advection field with magnitude in the range $\,0-300$, a
reference computation with $\,8.39\times10^6\,$ degrees of freedom yields the
result $\,\langle j,u^\varepsilon_{\text{ref}}\rangle\approx0.2170$. A
uniform sampling mesh with $\,32\,$ sampling regions is chosen, as well as a
macroscale discretization of $\,1.3\times10^5\,$ cells and a (fully resolved)
microscale discretization with $\,h=2^{-11}$. The optimization framework is
run for a fully resolved dual solution (``full'') with $\,2.1\times10^6\,$
cells as well as the reduced, locally enhanced variant given in
Definition~\ref{defi:4localenhancement} (``enhanced''). As stopping
criterion a reduction of $\,|\tilde\theta^\delta|\,$ to less than $\,5\,\%\,$ 
of the initial value is chosen with a penalty $\,\lambda=1.0\,m$, where $\,m\,$
is the absolute mean value of the diagonal entries of the matrix
$\,\mathcal{J}^T\mathcal{J}$, see (\ref{eq:5levenberg}), and a very small
regularization $\,\alpha_K=0.1\,$ (compared to $\,|\tilde\theta^\delta|^2/|\bar
A^\delta_K|^2\sim1\,000$). The numerical results are given in
Table~\ref{tab:5advdiffrandom}.

To examine the numerical stability of the optimization algorithm the
computation is actually run for $\,15\,$ adaptation cycles well beyond the
stopping criterion that is reached with step $\,12\,$ for the full dual solution
and with step $\,10\,$ for the local enhancement strategy. The initial error of
$~80\,\%$ in the target functional with a starting model $\,\bar
A^\delta_K=\gamma\,\text{Id}\,$ can be reduced to around $\,2-3\,\%\,$ for both
variants of dual solution. Further, the adaptation cycle remains stable
beyond the point where the stopping criterion was reached. Reference,
initial and final (for step $\,12\,$ and $\,10$, respectively) solutions 
are depicted in Figure~\ref{fig:5advdiffrandomsolution}. As can be seen 
from the numerical results, the microscale advection due to
$\,\boldsymbol{b^\varepsilon}\,$ leads to a locally increased value for $\,\bar
A^\delta_K\,$ in the range $\,0.01\,-\,0.02\,$ compared to the initial choice
$\,\bar A^\delta_K=\gamma\,\text{Id}\sim0.01$. The effective models found
with the optimization approach match the reference solution quite well near
the boundary $\,\Gamma_B$. In contrast, on the far end of $\,\Gamma_B\,$ near 
the inhomogeneous Neumann condition on $\,\Gamma_E$, the effective solutions
deviate from $\,u^\varepsilon$. This has to be expected as the optimization
problem only minimizes the error given by an integral over $\,\Gamma_B$.

\begin{table}
  \center
  \caption{Results for the model-optimization algorithm
    (Algorithm~\ref{alg:5modeloptimization}) applied to the
    advection-diffusion problem (\ref{eq:5advdiff}) with fully resolved
    dual solution (a) and for the reduced, locally enhanced variant (b).
    After steps $\,12\,$ and $\,10$, respectively, the estimator
    $\,|\tilde\theta^\delta|\,$ is reduced to less than $\,5\,\%$ of its initial
    value.}
  \label{tab:5advdiffrandom}
  \medskip
  \subfloat[model-optimization algorithm with fully resolved dual solution]{
    \footnotesize
    \begin{tabular} {r c c c c c c}
      \toprule
      & $L^2(\Omega)$
      & $|\langle j,U\rangle|$
      & $|\langle j,u^\varepsilon\!-\!U\rangle|$
      & $|\tilde\theta^\delta|$ & $I_{\text{eff}}$ & $I_{\text{loc}}$ \\[0.3em]
      \cmidrule(lr){2-2} \cmidrule(lr){3-5} \cmidrule(lr){6-7}
       1 & 4.43e-1 & 3.86e-1 & 1.69e-1 (77.9\,\%)  & 1.69e-1 & 1.00  & 2.21 \\
       3 & 2.80e-1 & 2.71e-1 & 5.36e-2 (24.7\,\%)  & 5.38e-2 & 1.00  & 4.24 \\
       5 & 2.29e-1 & 2.49e-1 & 3.16e-2 (14.6\,\%)  & 3.18e-2 & 1.00  & 5.86 \\
       7 & 1.90e-1 & 2.37e-1 & 2.02e-2 (9.30\,\%)  & 2.04e-2 & 1.00  & 7.68 \\
       9 & 1.60e-1 & 2.30e-1 & 1.32e-2 (6.08\,\%)  & 1.34e-2 & 1.00  & 10.2 \\
      11 & 1.39e-1 & 2.26e-1 & 9.14e-3 (4.21\,\%)  & 9.35e-3 & 1.00  & 13.2 \\[0.3em]
      12 & 1.30e-1 & 2.25e-1 & 7.91e-3 (3.64\,\%)  & 8.12e-3 & 1.00  & 14.7 \\[0.3em]
      15 & 1.13e-1 & 2.23e-1 & 6.28e-3 (2.89\,\%)  & 6.51e-3 & 0.99  & 17.3 \\
      \bottomrule
    \end{tabular}
  }

  \subfloat[model-optimization algorithm with reduced, locally enhanced
    dual solution]{
    \footnotesize
    \begin{tabular} {r c c c c c c}
      \toprule
      & $L^2(\Omega)$
      & $|\langle j,U\rangle|$
      & $|\langle j,u^\varepsilon\!-\!U\rangle|$
      & $|\tilde\theta^\delta|$ & $I_{\text{eff}}$ & $I_{\text{loc}}$ \\[0.3em]
      \cmidrule(lr){2-2} \cmidrule(lr){3-5} \cmidrule(lr){6-7}
       1 & 4.43e-1 & 3.86e-1 & 1.69e-1 (77.9\,\%) & 2.96e-1 &  1.76 & 2.15 \\
       3 & 2.33e-1 & 2.62e-1 & 4.48e-2 (20.6\,\%) & 4.92e-2 &  1.10 & 4.01 \\
       5 & 1.46e-1 & 2.39e-1 & 2.20e-2 (10.1\,\%) & 2.85e-2 &  1.30 & 3.85 \\
       7 & 9.86e-2 & 2.28e-1 & 1.08e-2 (4.97\,\%) & 2.06e-2 &  1.90 & 3.80 \\
       9 & 8.21e-2 & 2.22e-1 & 5.22e-3 (2.40\,\%) & 1.58e-2 &  3.04 & 4.20 \\[0.3em]
      10 & 7.99e-2 & 2.20e-1 & 3.02e-3 (1.39\,\%) & 1.27e-2 &  4.18 & 4.91 \\[0.3em]
      15 & 9.70e-2 & 2.14e-1 & 3.07e-3 (1.41\,\%) & 7.00e-3 &  2.28 & 10.4 \\
      \bottomrule
    \end{tabular}
  }
\end{table}

\begin{figure}
  \centering
  \subfloat[]{
    \quad
    \includegraphics[height=4.0cm] {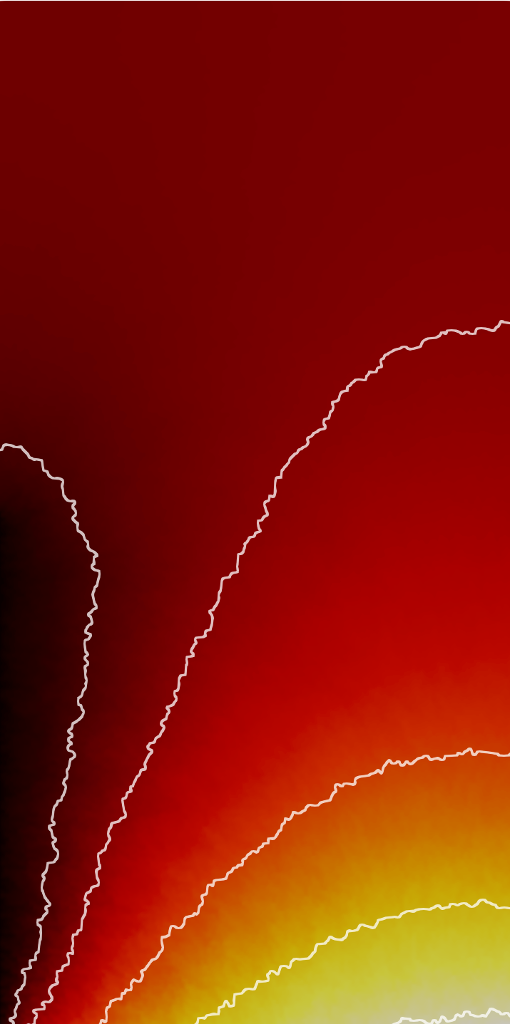}
    \quad
  }
  \subfloat[]{
    \quad
    \includegraphics[height=4.0cm] {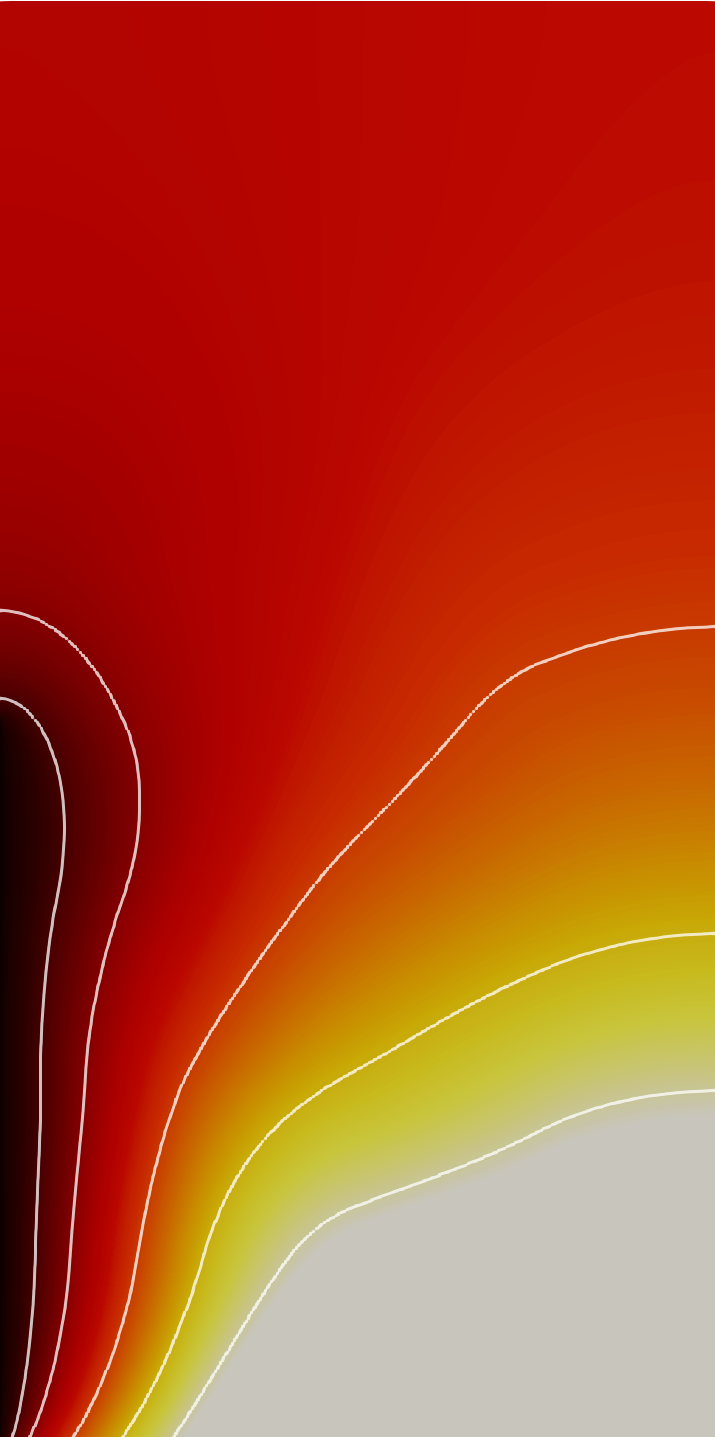}
    \quad
  }
  \subfloat[]{
    \quad
    \includegraphics[height=4.0cm] {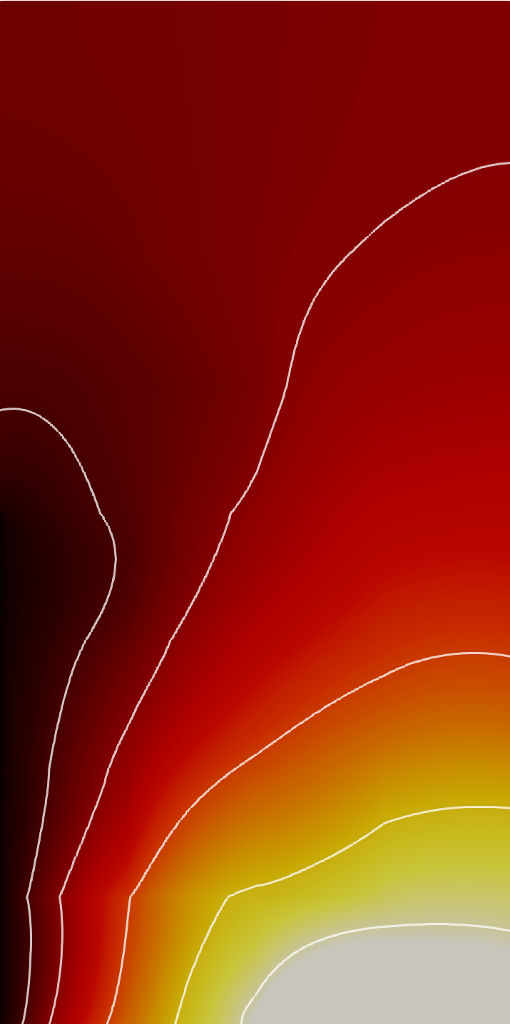}
    \quad
  }
  \includegraphics[height=4.0cm] {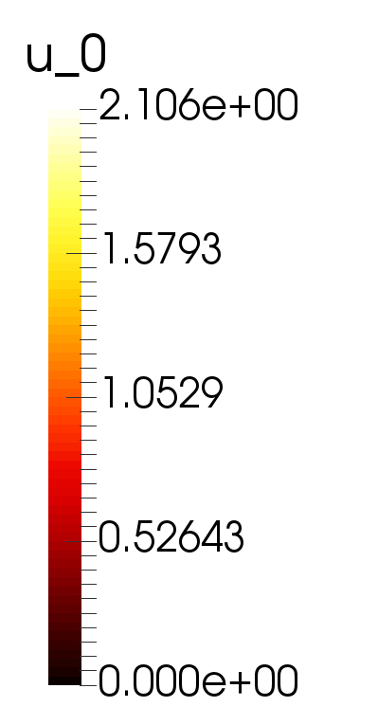}
  \caption{Plot of the reference solution (a), the initial
    solution $u^{\delta,0}$ of the optimization problem (b), and the final
    solution $u^{\delta,10}$ (c) obtained for the reduced, locally enhanced
    dual solution.}
  \label{fig:5advdiffrandomsolution}
\end{figure}


\section{Conclusion}
\label{sec:6conclusion}

A novel approach for model adaptivity is proposed that is based on solving
an optimization problem by minimizing the local model-error indicators
derived from a DWR formulation. The optimization approach allows to derive
an efficient post-processing strategy that can be regarded as a multiscale
approach in its own right. Its strength lies in the fact that it is in
principle independent of strong a priori knowledge about applicability of
efficient models---its efficiency is rooted in the almost quantitative
behavior of the DWR method when combined with a suitable localization
technique for the dual problem. The modeling aspect of the optimization
problem lies in the choice of the functional $\,\langle j,\cdot\rangle\,$ 
as quantity of interest (given by the application in mind) and
the choice of the localization approach for the dual problem. In this sense it
lifts the question of suitable approximation in terms of a quantity of
interest (for the primal problem) to the question of suitable approximation
properties of the localization technique for the dual problem. The
important property here is that the latter is typically measured in the
$L^2$-norm of the gradient of the error of the dual approximation, for
which---depending on the localization approach---strong approximation
properties are available.
Prototypical numerical results are presented for a heterogeneous elliptic
diffusion and an advection-diffusion problem, that indicate that the
optimization approach combined with a localization technique that
\emph{globally} uses the same effective model as the primal problem and
\emph{locally reconstructs} finescale features of the full dual solution
does result in an efficient model-adaptation strategy.


\end{document}